\documentclass[a4paper,11pt,leqno]{amsart}
\usepackage[applemac]{inputenc}
\usepackage{epsfig,graphics}
\usepackage{amsthm,amssymb,amsbsy,a4wide}
\usepackage{bbm}
\usepackage{longtable}

\def\Gal{\mathrm{Gal}}
\def\kk{\mathbbm{k}}
\def\Z{\mathbbm{Z}}

\def\F{\mathbbm{F}}
\def\C{\mathbbm{C}}
\def\Q{\mathbbm{Q}}
\def\un{\mathbbm{1}}
\def\tr{\mathrm{tr}}
\def\End{\mathrm{End}}
\def\Ker{\mathrm{Ker}}
\def\Aut{\mathrm{Aut}}

\def\Id{\mathrm{Id}}

\def\la{\lambda}
\def\ii{\mathrm{i}}

\def\eps{\varepsilon}

\def\onto{\twoheadrightarrow}

\def\GL{\mathrm{GL}}
\def\PGL{\mathrm{PGL}}
\def\SL{\mathrm{SL}}
\def\PSL{\mathrm{SL}}

\def\SU{\mathrm{SU}}

\def\PSL{\mathrm{PSL}}

\def\SP{\mathrm{SP}}
\def\OSP{\mathrm{OSP}}
\def\Spin{\mathrm{Spin}}

\def\diag{\mathrm{diag}}
\def\Imm{\mathrm{Im}\,}

\newtheorem{theor}{Theorem}
\newtheorem{lemma}[theor]{Lemma}
\newtheorem{prop}[theor]{Proposition}

\numberwithin{theor}{section}

\title{Image of the braid groups inside the finite Iwahori-Hecke algebras}
\author{Olivier Brunat}
\author{Kay Magaard}
\author{Ivan Marin}
\address{Institut de Math\'ematiques de Jussieu \\ Universit\'e Paris 7 \\ 175 rue du Chevaleret \\ 75013 Paris \\ France }
\email{brunat@math.jussieu.fr}
\address{
School of Mathematics\\
University of Birmingham\\
Edgbaston\\
Birmingham B15 2TT\\
United Kingdom} \email{k.magaard@bham.ac.uk}
\address{LAMFA \\ Universit\'e de Picardie-Jules Verne \\ 33 rue Saint-Leu \\ 80039 Amiens Cedex 1 \\ France }
\email{ivan.marin@u-picardie.fr} 

\date{January 3, 2014.}
\begin{document}

\begin{abstract}
We determine the image of the braid groups inside the Iwahori-Hecke algebras of type A, when
defined over a finite field, in the semisimple case, and for
suitably large (but controlable) order of the defining (quantum) parameter. 
\end{abstract}

\maketitle

\section{Introduction}

The point of this paper is to enhance our understanding of the connection between braid groups and
Hecke algebras of type A.
This interplay has been at the core of the definition of the Jones and subsequently HOMFLYPT polynomial
of knots and links,
and is the source of the most classical linear representations of the braid groups. Because of that, it has
also been used for the purpose of inverse Galois theory -- in that case, with coefficients a finite field.
Our aim here is to understand better the image of the braid group inside the (group of invertible elements of) the Hecke algebra, 
and especially to describe the finite group which is the image of the braid group inside the Hecke algebra
over a finite field. We first review briefly what is known.

The closed image of the braid group inside the Hecke algebra over
the complex numbers has been essentially determined in the first decade of the
century. In this setting, it had been proved earlier by Jones and Wenzl
that the Hecke algebra representations provided unitary
representations of the braid group for suitable parameters. Using this,
the closed image in these unitary cases was determined in \cite{FLW}. Simultaneously and independently,
the third author in his 2001 doctoral thesis (see \cite{TH}), introduced a Lie algebra subsequently identified (see \cite{LIETRANSP}) with the
Lie algebra of the algebraic closure of $B_n$ in the generic (but not necessarily unitary) case. When the representation
is known to be unitary, the algebraic closure determines the topological closure. On the other hand,
the approach of \cite{FLW} provides more precise information on specific values of the papers, specifically
when the parameter is a root of $1$. Finally, other proofs and sources of justification, sometimes in a broader
context, for the unitary structures have been provided in \cite{GTcompo} and \cite{IH}, part IV.

Back to the finite field situation, the classical ``strong approximation'' results suggest that, ``most of the time'', we should
get for images groups of $\F_q$-points of the algebraic groups defined above. This assertion is very vague
because the algebraic groups are not a priori defined over $\Z$ and because there is a parameter
involved in the definition of the Hecke algebra that prevents the direct use of these classical results. Also, there is the question of unitarity
which needs some work to be translated into the finite fields case. Nevertheless, the first and third author proved in \cite{BM} 
that we can get the expected result for the quotient of the Hecke algebra known as the Temperley-Lieb
algebra, under only a few conditions, the most restrictive of these being that the corresponding
Hecke algebra is semisimple. By classical results from representation theory this last condition can be made
precise in terms of the order of the parameter inside $\F_q^{\times}$ and in terms of the number $n$ of strands.

In this paper we extend this to the full Hecke algebra, under the same conditions.
For technical reasons we found it more handy to deal with the commutator subgroup $\mathcal{B}_n$ of $B_n$ instead
of $B_n$ itself. Since $B_n^{ab} \simeq \Z$ this does not diminish the strength of the
results, and at the same time makes many proofs and statements more readable. 

We now state the main result. 
We let $\mathcal{E}_n$ denote the set of partitions on $n$ which are not hooks. We choose some total ordering $<$ on $\mathcal{E}_n$.
Let $b(\la) = \max \{ i ; \la_i \geq i \}$ denote the length of the diagonal of the Young diagram associated to $\la \vdash n$,
and $\nu(\la)=1$ if  $(n - b(\la))/2$ is even, $\nu(\la) = -1$ otherwise. Without loss of generality we assume that $\F_q = \F_p(\alpha)$
and denote $\GL(\la)$ the group of linear automorphisms of the $\F_q$-vector space associated to the representation of $H_n(\alpha)$ indexed by $\la$.
Because of the existence of explicit matrix models recalled below, we know that these representations are indeed defined over $\F_q = \F_p(\alpha)$.
In \S 3 we attach to each $\la \in \mathcal{E}_n$ a classical subgroup $G(\la)$ of $\GL(\la)$ which contains the
image of $\mathcal{B}_n$. Letting $N$ denote the dimension of the representation
attached to $\la$, we have the following, where we use the classical notations
of e.g. \cite{WILSON}. In particular $\Omega_N^+(q)$ is the commutator subgroup of the
orthogonal group for a form of `$+$' type, meaning that is has Witt index $0$. 
\begin{itemize}
\item If $p = 2$, then
\begin{itemize}
\item if $\F_2(\alpha+\alpha^{-1}) = \F_q$,
\begin{itemize}
\item if $ \la \neq \la'$, $G(\lambda) = \SL_N(q)$
\item if $ \la = \la'$, then  $G(\lambda) = \SP_N(q) $
\end{itemize}
\item if $\F_2(\alpha+\alpha^{-1}) \neq \F_q$, then $\F_2(\alpha+\alpha^{-1}) = \F_{\sqrt{q}}$ and
\begin{itemize}
\item if $ \la \neq \la'$, $G(\lambda) = \SU_N(q) $
\item if $ \la = \la'$, then  $G(\lambda) = \SP_N(\sqrt{q}) $
\end{itemize}
\end{itemize}
\item If $p$ is odd, then
\begin{itemize}
\item if $\F_p(\alpha+\alpha^{-1}) = \F_q$,
\begin{itemize}
\item if $ \la \neq \la'$, $G(\lambda) = \SL_N(q)$
\item if $ \la = \la'$ and $\nu(\lambda) = -1$, then  $G(\lambda) = \SP_N(q) $
\item if $ \la = \la'$ and $\nu(\lambda) = 1$, then  $G(\lambda) = \Omega_N^+(q)$
\end{itemize}
\item if $\F_p(\alpha+\alpha^{-1}) \neq \F_q$, then $\F_p(\alpha+\alpha^{-1}) = \F_{\sqrt{q}}$ and
\begin{itemize}
\item if $ \la \neq \la'$, $G(\lambda) = \SU_N(q) $
\item if $ \la = \la'$ and $\nu(\lambda) = -1$, then  $G(\lambda) = \SP_N(\sqrt{q}) $
\item if $ \la = \la'$ and $\nu(\lambda) = 1$, then  $G(\lambda) =\Omega_N^+(\sqrt{q}) $
\end{itemize}
\end{itemize}
\end{itemize}
We recall that the Hecke algebra $H_n(\alpha)$ for $\alpha \in \F_q^{\times}$ can be defined as the quotient of the group algebra $\F_q B_n$ of the braid group $B_n$
by the relations $(\sigma_i +1)(\sigma_i - \alpha) = 0$, where the $\sigma_i$ are the usual Artin generators of $B_n$. The algebra $H_n(\alpha)$ is semisimple when
the order of $\alpha$ is greater than $n$, and this provides an isomorphism $H_n(\alpha)^{\times} \simeq  \prod_{\la \vdash n} \GL(\la)$.

Then, our main theorem states the following.
\begin{theor} \label{maintheo} Assume $\F_q = \F_p(\alpha)$ and that the order of $\alpha$ is $> n$ and not $2,3,4,5,6,10$.
The morphism $\mathcal{B}_n \to H_n(\alpha)^{\times} \simeq \prod_{\la \vdash n} \GL(\la)$ factorizes though the morphism
$$
\Phi_n : \mathcal{B}_n  \to G(\lambda^0) \times \prod_{\stackrel{\la \in \mathcal{E}_n}{\la < \la' }} G(\la) \times \prod_{\stackrel{\la \in \mathcal{E}_n}{\la = \la' }} G(\la) 
$$
where $\la^0 = [n-1,1]$.
\end{theor}

The additional condition, that the order of $\alpha$ is not $2,3,4,5,6,10$, was expected,
for the image of $B_3$ in these cases
may in general factorize through the quotients of $B_3$ by the relations $\sigma_i^r = 1$ for $r \in \{2,3,4,5 \}$, which are imprimitive
reflection groups of rank 2 (see \cite{COXETER}).

We explain the plan of the proof. The price for using the commutator subgroup instead of the full braid group is that we need a few additional technicalities
that we gather in \S 2. The first step of the actual proof is then to get a description of the algebraic groups involved here in very explicit terms. For this we use Hoefsmit's combinatorial
matrix models in order to define the expected orthogonal and symplectic forms as well as the expected diagonal embeddings (see \S 3). Using
the vanishing of the Brauer group of the finite fields we show how to convert the unitarity property
into a well-defined algebraic group over a smaller field (\S 4). Then we proceed by an induction argument (\S 5) in order
to prove that the image of $\mathcal{B}_n$ is what we expect it is. By \cite{BM} we know it for $n \leq 5$. We first show that, assuming the result for some $n\geq 5$, we can determine the image of $ \mathcal{B}_{n+1}$ inside
every single irreducible representation of the Hecke algebra. For this, our main tool is a theorem of Guralnik and Saxl
on subgroups of finite classical groups acting irreducibly on the underlying vector space (notice that this theorem depends on
the classification theorem of finite simple groups). 
Then, as in \cite{BM},
we glue the pieces together in order to get the result for $n+1$ using Goursat's lemma. Finally, we indicate
how the proof needs to be modified in case the order of the parameter implies that a unitary structure is involved.

Generalizations of this work can be expected in two directions. One of them is to look at what happens
for the generalized braid groups associated to other (real or complex) reflection groups. The ``generic
image'', that is the Zariski closure over a field of characteritic 0 and for the generic values
of the parapeters, has been computed in \cite{IH}. Moreover, the unitarity property has
been proved for all Coxeter groups and most of the complex reflection groups, and is conjectured
to hold in general (see \cite{IH}, part III, \S 6). However the interplay between the unitarity property and the
algebraic structure, when looked at carefully, presents some additional difficulties for complex reflection groups,
see \cite{IH}, part IV, \S 5 and remark 5.9 there. When the reflection group if not rational, there is moreover a specialization issue,
because the base ring $\Z[q,q^{-1}]$ needs to be replaced by a ring of Laurent polynomials over a larger
ring of algebraic integers. Finally, for exceptional complex reflection groups, even the basic
structure theorems for the Hecke algebra are still conjectural (see \cite{JPAA} for an overview and recent results).
Even in the Coxeter case, quite a few tools we used here however cannot be applied directly in the
more general context. Moreover, the Hecke algebras may involve several parameters, and also because of that
the unitarity property may be more tricky to handle. As an example of what may happen, let us mention that the image
of the generalized braid goup of Coxeter type $H_4$ should be quite interesting, because the representations
of the reflection groups can be defined only over $\Q(\sqrt{5})$, and because there is a $\Spin_8$ group appearing
in the description of the generic image.

A second natural direction is to try to understand what happens in the non-semisimple case,
that is when the order of $\alpha$ is lower or equal to $n$. As far as we know, this is yet a completely unexplored
territory, also over the complex numbers when $\alpha$ is a root of $1$.

\section{Preliminaries on braid groups}

We let $\mathcal{B}_n$ denote the commutator subgroup of the braid group $B_n$ on $n$ strands, and always
identify $B_{n-1}$ with the subgroup of $B_n$ fixing the last strand.

\begin{lemma} \label{lemnormgen}
If $n \geq 4$ then $\mathcal{B}_n$ is the normal closure of $\mathcal{B}_{n-1}$.
\end{lemma}
\begin{proof}
Recall that the abelianization morphism $\ell : B_n \to \Z$ is given by $s_i \mapsto 1$. From the
Reidemeister-Schreier method or even elementary group theory we know that $\mathcal{B}_n$ is generated by the $s_1^k s_j s_1^{-k-1}$
for $j \geq 1$, $k \in \Z$. When $j > 2$ we have $s_1^k s_j s_1^{-k-1} = s_j$, which proves that $\mathcal{B}_n$
is generated by $\mathcal{B}_{n-1}$ and $s_n s_1^{-1}$.
Now the braid relation $s_{n-1} s_n s_{n-1} = s_n s_{n-1} s_n$ implies
$s_n = s_{n-1} s_n s_{n-1}(s_{n-1} s_n)^{-1}$
hence
$$
s_n s_1^{-1}= (s_{n-1} s_n) s_{n-1}s_1^{-1}(s_{n-1} s_n)^{-1}
= (s_{n-1}s_1^{-1} s_ns_1^{-1}) s_{n-1}s_1^{-1}(s_{n-1}s_1^{-1} s_ns_1^{-1})^{-1}
$$
belongs to the normal closure of $\mathcal{B}_{n-1}$
and this proves the claim.
\end{proof}

In order to use known representation-theoretic results for the braid group, we shall need to lift isomorphisms
between the restrictions of these representations to $\mathcal{B}_n$. This will be done by applying the
following general lemma.

\begin{lemma} \label{lem:gprime} Let $G$ be a group, $\kk$ a field and $R_1,R_2 : G \to \GL_N(\kk)$
with $N \geq 2$ two representations, such that $(R_1)_{|G'} =
(R_2)_{|G'}$, where $G'$ denotes the commutator subgroup of $G$,
and such that the restriction of $R_2$ to $G'$ is absolutely irreducible.
 Then
there exists a character $\eta : G \to \kk^{\times}$ such
that $R_2 = R_1 \otimes \eta$.
\end{lemma}
\begin{proof} Let $\eta : G \to \GL_N(\kk)$ the
map defined by $\eta(g) = R_2(g) R_1(g)^{-1}$. 
For $g \in G$ and $h \in G'$,
we have $\eta(gh) = R_2(g) (R_2(h) R_1(h)^{-1}) R_1(g)^{-1}
=  R_2(g) R_1(g)^{-1} = \eta(g)$,
and also $\eta(gh) = \eta(ghg^{-1}.g) 
= R_2(ghg^{-1}) R_2(g) R_1(g)^{-1} R_1(ghg^{-1})
= R_2(ghg^{-1}) \eta(g) R_2(ghg^{-1})$.
It follows that $\eta(g)$
centralizes $R_2(g G' g^{-1}) = R_2(G')$. By the absolute irreducibility
assumption and Schur's lemma we get that
$\eta(g) \in \kk^{\times}$. Then
$\eta(g_1 g_2) = R_2(g_1) (R_2(g_2) R_1(g_2)^{-1})R_1(g_1)^{-1}   
= R_2(g_1) \eta(g_2) R_1(g_1)^{-1}    
= R_2(g_1)  R_1(g_1)^{-1}    \eta(g_2)
= \eta(g_1)    \eta(g_2)$
for all $g_1,g_2 \in G$, which proves the claim.

\end{proof}

We shall also use the following result.

\begin{prop} Let $K$ be a field, $\varphi : B_n \to \PSL_2(K)$ an homomorphism with $n \geq 5$. Then
$\varphi(B_n)$ is abelian (and therefore cyclic).
\end{prop}
\begin{proof} Without loss of generality we can assume that $K$ is algebraically closed.
Let $S_i = \varphi(s_i)$. If one of the $S_i$ is 1, the same holds for the others since they
are all conjugated one to the other, hence $\varphi = 1$. Also note that if
two consecutive $S_i$ commute, then the braid
relation implies $S_i = S_{i+1}$, and this implies that all the $S_i$ are equal,
and therefore that $\varphi(B_n)$ is abelian. This is because every pair $(s_i, s_{i+1})$ is easily seen
to be conjugated to any other pair $(s_j,s_{j+1})$ by an element of $B_n$.

 Let $\ii$ denote a primitive $4$-root of $1$,
and $E$ the image of $\begin{pmatrix} \ii & 0 \\ 0 & -\ii \end{pmatrix} \in \SL_2(K)$ inside $\PSL_2(K)$.
We let $T$ denote the images of the diagonal matrices of determinant $1$ inside $\PSL_2(K)$,
and $T'$ the images of the antidiagonal matrices.
We first assume that $S_1$ is semisimple. Then all the $S_i$
are semisimple. Up to conjugation,
we can assume that $S_1 \in T$. Then the centralizer of $S_1$ is $T$, unless $S_1 = E$
in which case it is $T \cup T'$. If the centralizer is $T$, then $S_3, S_4 \in T$ and we get 
$S_3 S_4 = S_4 S_3$
and this implies that $\varphi(B_n)$ is abelian. If not, we have $S_1 = E$ hence $S_1^2 =1$
and therefore $S_i^2 = 1$ for all $i$. It follows that $\varphi$ factorizes through a morphism
$\mathfrak{S}_n \to \PSL_2(K)$. 
If $\varphi(B_n)$ is not abelian the morphism $\mathfrak{S}_n \to \PSL_2(K)$
is into. But for $n \geq 5$ this contradicts Dickson's theorem (see e.g. \cite{SUZUKI} ch. 3 theorem 6.17).
This proves the statement under the assumption that $S_1$ is semisimple.

If not, $S_1$ is unipotent and we can assume that $S_1$ is upper triangular. Then its centralizer
is made of the image inside $\PGL_2(K)$ of upper-triangular matrices. It follows that $S_3$ and $S_4$
commute, and we conclude as before.

\end{proof}

The statement we are mostly interested in is the following one.

\begin{prop} Let $K$ be a field. If $n \geq 7$ and  $\varphi : \mathcal{B}_n \to \PSL_2(K)$ is an homomorphism,
$\varphi = 1$.
\end{prop}
\begin{proof} A presentation of $\mathcal{B}_n$ has been obtained by Gorin-Lin in \cite{GORINLIN}, Theorem 2.1. 
We use it here.
The group $\mathcal{B}_n$ is generated by elements $p_0 (= s_2 s_1^{-1})$, $p_1(= s_1 s_2 s_1^{-2})$,
$b (=s_2 s_1^{-1} s_3 s_2^{-1})$, $q_{\ell} (= s_{\ell} s_1^{-1})$ for $3 \leq \ell \leq n-1$; and
relations
$$
\begin{array}{llllll}
(1) & b = p_0 q_3 p_0^{-1} & (2) & p_0 b p_0^{-1} = b^2 q_3^{-1} b & (3) & p_1 q_3 p_1^{-1} = q_3^{-1} b \\
(4) & p_1 b p_1^{-1} = (q_3^{-1} b)^3 q_3^{-2} b & (5) & p_0 q_i = q_i p_1 (i \geq 4) & (6) & p_1 q_i = q_i p_0^{-1} p_1 (i \geq 4)  \\
& & (7) & q_i q_{i+1} q_i = q_{i+1} q_i q_{i+1} ( i \geq 3) & (8) & q_i q_j = q_j q_i (i\geq3, j >i+1)
\end{array}
$$
By abuse of notation, we identity these generators with their images under $\varphi$, and we show that they
all become trivial. First note that, if one of the $q_i$ is $1$, then all the others are equal to $1$ by relation (7), and then $b = 1$ by (3), $p_0 = p_1$ by (5) and $p_0 = 1$ by (6).
Conversely,  $p_0=1 \Leftrightarrow p_1 = 1$ by (5), and in this case $b=q_3$ by (1), and (2) implies $b^2 = b$ hence $b=q_3 = 1$.
Finally, if $b=1$, then $q_3 = 1$ by (1).

Now note that we have a morphism $B_{n-2} \to \mathcal{B}_n$ defined by $s_i \mapsto q_{i+2}$. By the above proposition
we get that the $q_i,i \geq 3$ commute one to the other, and therefore are all equal to some element $q$.
$$
p_1 q p_1^{-1} \stackrel{(3)}{=} q^{-1} b  \stackrel{(1)}{=} q^{-1} \underline{p_0 q} p_0^{-1} \stackrel{(5)}{=} q^{-1} q p_1 p_0^{-1}
$$
hence $p_1 q p_1^{-1} = p_1 p_0^{-1}$ hence $p_1 q^{-1} = p_0$ and
$p_1 = p_0 q \stackrel{(5)}{=} q p_1$ hence $q = 1$, a contradiction which proves the claim.

\end{proof}

\section{The main factorisation}
\label{sectmainfactor}

We recall that $H_n(\alpha)$ is semisimple as soon as the order of $\alpha \in \F_q^{\times}$ is greater
than $n$. Moreover, in this case its simple modules are absolutely semisimple (see e.g. \cite{MATHAS}, cor. 3.44), and
they are in $1-1$ correspondence with the partitions of $n$. We now recall from \cite{GP} explicit
matrix models for these irreducible representations.

A combinatorial Gelfand model of $H_n(\alpha)$ is given by a $\F_q$-vector space $\mathcal{V}$
with basis all the standard tableaux of size $n$. For each partition $\la \vdash n$, we denote
$V_{\la}$ the linear span of the standard tableaux of shape $\la$.

The action of the $r$-th generator on a standard tableau $\mathbbm{T}$ is given by the following rules

\begin{enumerate}
\item If $r$ and $r+1$ lie in the same row of $\mathbbm{T}$, then $s_r  . \mathbbm{T} = \alpha \mathbbm{T}$ ;
\item if $r$ and $r+1$ lie in the same column of $\mathbbm{T}$, then $s_r  . \mathbbm{T} = - \mathbbm{T}$ ;
\item otherwise, $s_r . \mathbbm{T} = m_r(\mathbbm{T}) \mathbbm{T} + (1 + m_r(\mathbbm{T})) \mathbbm{T}_{r \leftrightarrow r+1}$,
where $$ m_r(\mathbbm{T}) =  \frac{(\alpha-1) ct(\mathbbm{T} : r+1)}{ct(\mathbbm{T} : r+1) - ct(\mathbbm{T} : r)},$$
{}$ct(\mathbbm{T} : m) = -\alpha^{j-i}$ if $m$ is in line $i$ and column $j$, and $\mathbbm{T}_{r \leftrightarrow r+1}$
is the tabeau obtained from $\mathbbm{T}$ by interchanging $r$ and $r+1$.
\end{enumerate}
Notice that ($\mathbbm{T}_{r \leftrightarrow r+1})' = (\mathbbm{T}')_{r \leftrightarrow r+1}$. Moreover, if we let $i$ denote the row, $j$ denote the column where $r$ lies, and similarly $u,v$ for $r+1$,
one checks easily that
$$
m_r(\mathbbm{T}') = m_r(\mathbbm{T}_{r \leftrightarrow r+1}) = - \alpha^{j-i+u-v} m_r(\mathbbm{T})
$$
where $\mathbbm{T}'$ denotes the transposed of  $\mathbbm{T}$.
We define a bilinear form on $\mathcal{V}$ by the formula $(\mathbbm{T}_1|\mathbbm{T}_2) = w(\mathbbm{T}_1) \delta_{\mathbbm{T}_2,\mathbbm{T}_1'}$ where 
$$
w(\mathbbm{T}) = \prod_{\stackrel{ i<j}{r_i(\mathbbm{T}) > r_j(\mathbbm{T})}} (-1) = (-1)^{ \# \{ i< j \ | \ r_i(\mathbbm{T}) > r_j(\mathbbm{T}) \} }
$$
and $r_k(\mathbbm{T})$ denotes the row of $\mathbbm{T}$ in which lies $k$.
\begin{prop} Let $b(\la) = \max \{ i ; \la_i \geq i \}$ denote the length of the diagonal of the Young diagram associated to $\la \vdash n$,
and $\nu(\la)=1$ if  $(n - b(\la))/2$ is even, $\nu(\la) = -1$ otherwise.
\begin{enumerate}
\item For all $b \in B_n$, we have $(b. \mathbbm{T}_1 | b. \mathbbm{T}_2) = (-\alpha)^{\ell(b)}(\mathbbm{T}_1|\mathbbm{T}_2)$ for any standard tableaux $\mathbbm{T}_1, \mathbbm{T}_2$.
\item For all $b \in \mathcal{B}_n$, we have $(b. \mathbbm{T}_1 | b. \mathbbm{T}_2) = (\mathbbm{T}_1|\mathbbm{T}_2)$ for any standard tableaux $\mathbbm{T}_1, \mathbbm{T}_2$.
\item The restriction of the bilinear form $(\ | \ )$ to subspaces $V_{\la}$ if $\la = \la'$ and $V_{\la} \oplus V_{\la'}$ if $\la \neq \la'$ is nondegenerate. Its restriction to $V_{\la}$  is symmetric if $\nu(\la)=1$, and skew-symmetric otherwise. When it is symmetric, it has Witt index $0$.
\end{enumerate}
\end{prop}
\begin{proof} In order to prove
(i) and (ii), we check that $(s_r .\mathbbm{T}_1 | s_r . \mathbbm{T}_2) = (-\alpha) (\mathbbm{T}_1| \mathbbm{T}_2)$ for all $r$.
If $r$ and $r+1$ lie in the same row or the same column of $\mathbbm{T}_1$, the LHS and RHS are both $0$
unless $\mathbbm{T}_2 = \mathbbm{T}_1'$, and in that case the verification of the formula is immediate. If not, the LHS and RHS are again
both $0$, except in  two cases that we consider separately. In the first one, we have $\mathbbm{T}_1 = \mathbbm{T}$, $\mathbbm{T}_2 = \mathbbm{T}'$.
In that case we have
 $(s_r .\mathbbm{T}_1 | s_r. \mathbbm{T}_2) = (s_r. \mathbbm{T}|s_r \mathbbm{T}')$ and, since
$s_r . \mathbbm{T}' = m_r(\mathbbm{T}') \mathbbm{T}' + (1 + m_r(\mathbbm{T}')) \mathbbm{T}'_{r \leftrightarrow r+1}$,
 we get
 $$(s_r.\mathbbm{T} | s_r.\mathbbm{T}') = m_r(\mathbbm{T})m_r(\mathbbm{T}') w(\mathbbm{T}) + (1+m_r(\mathbbm{T}))(1+m_r(\mathbbm{T}')) w(\mathbbm{T}_{r \leftrightarrow r+1})
 $$
 and $(s_r. \mathbbm{T} | s_r \mathbbm{T}') = - \alpha (\mathbbm{T} | \mathbbm{T}')$ iff
 $$
 \left(-\alpha - m_r(\mathbbm{T}) m_r(\mathbbm{T}') \right) w(\mathbbm{T}) = (1+m_r(\mathbbm{T}))(1+m_r(\mathbbm{T}')) w(\mathbbm{T}_{r \leftrightarrow r+1}) 
 $$
In the other case we have $\mathbbm{T}_1 = \mathbbm{T}$, $\mathbbm{T}_2 = \mathbbm{T}_{r \leftrightarrow r+1}'$. 
In this case
 $(s_r . \mathbbm{T}_1 | s_r . \mathbbm{T}_2) = (s_r. \mathbbm{T}| s_r . \mathbbm{T}_{r \leftrightarrow r+1}')$ and, since
 $s_r . \mathbbm{T}_{r \leftrightarrow r+1}' = m_r(\mathbbm{T}_{r \leftrightarrow r+1}') \mathbbm{T}_{r \leftrightarrow r+1}'
 + (1+m_r(\mathbbm{T}_{r \leftrightarrow r+1}')) \mathbbm{T}'$ we get that $(s_r.\mathbbm{T}_1 | s_r. \mathbbm{T}_2)$
is 
 $$m_r (\mathbbm{T})(1+m_r(\mathbbm{T}'_{r \leftrightarrow r+1}) )w(\mathbbm{T}) + (1 + m_r(\mathbbm{T})) m_r(\mathbbm{T}'_{r \leftrightarrow r+1})
 w(\mathbbm{T}_{r \leftrightarrow r+1})
$$
and $(s_r. \mathbbm{T}_1 | s_r \mathbbm{T}_2) = - \alpha (\mathbbm{T}_1 | \mathbbm{T}_2)=0$ iff
$$
  - m_r (\mathbbm{T})(1+m_r(\mathbbm{T}'_{r \leftrightarrow r+1})) w(\mathbbm{T})  = (1 + m_r(\mathbbm{T})) m_r(\mathbbm{T}'_{r \leftrightarrow r+1})
 w(\mathbbm{T}_{r \leftrightarrow r+1})
$$
By a direct computation we check that
$$
\frac{- m_r (\mathbbm{T})(1+m_r(\mathbbm{T}'_{r \leftrightarrow r+1}))}{ (1 + m_r(\mathbbm{T})) m_r(\mathbbm{T}'_{r \leftrightarrow r+1})}
=
\frac{\left(-\alpha - m_r(\mathbbm{T}) m_r(\mathbbm{T}') \right)}{ (1+m_r(\mathbbm{T}))(1+m_r(\mathbbm{T}')) } = -1
$$
hence the equations hold in both cases because of the elementary properties of $w$, namely $w(\mathbbm{T}_{r \leftrightarrow r+1}) = - w(\mathbbm{T})$.

 We now prove (iii). The non-degeneracy of $(\ | \ )$ follows from the
decomposition of $V_{\la}$ as an orthogonal direct sum of planes spanned by pairs $\mathbbm{T}, \mathbbm{T}'$, on which
$(\ | \ )$ is clearly non-degenerate. We consider now the possible symmetry of the restriction of $(\ | \ )$ to some $V_{\la}$
with $\la = \la'$.
We proved in \cite{LIETRANSP}, Lemme 6, that $w(\mathbbm{T}) w(\mathbbm{T}')$ only depends on the shape
 $\la$ of $\mathbbm{T}$, and is equal to $\nu(\la)$.
 Since $$(\mathbbm{T}_2|\mathbbm{T}_1) = 
 w(\mathbbm{T}_2) \delta_{\mathbbm{T}_1,\mathbbm{T}_2'} = 
 w(\mathbbm{T}_2) \delta_{\mathbbm{T}_2,\mathbbm{T}_1'} = 
\frac{1}{ \nu(\la)} w(\mathbbm{T}_2') \delta_{\mathbbm{T}_2,\mathbbm{T}_1'} = 
\nu(\la) w(\mathbbm{T}_1) \delta_{\mathbbm{T}_2,\mathbbm{T}_1'} =
\nu(\la) (\mathbbm{T}_1|\mathbbm{T}_2)
$$
we get the conclusion.
Finally, the computation of the Witt index in the symmetric case is an immediate consequence of the
direct sum decomposition in hyperbolic planes already mentioned.
\end{proof}

We define $\mathcal{L} \in \End(\mathcal{V})$ by
$ \mathbbm{T} \mapsto w(\mathbbm{T}) \mathbbm{T}'$. We have
\begin{lemma} Let $\la \vdash n$ such that $\la \neq \la'$. Then $\mathcal{L}$ induces an endomorphism
of $V_{\la} \oplus V_{\la'}$ exchanging $V_{\la}$ and $V_{\la'}$ such that the action of $s_r$
satisfies 
$$
\frac{\mathcal{L} s_r \mathcal{L}^{-1}}{(- \alpha)\nu(\la)} =    ^t s_r^{-1}
$$
\end{lemma}
\def\lra{\leftrightarrow}
\begin{proof}
We check that the actions of the LHS and RHS coincide on every standard tableau $\mathbbm{T}$ of shape $\la$. When $s_r. \mathbbm{T}$ is proportional to $\mathbbm{T}$, this directly follows from the formula $w(\mathbbm{T})w(\mathbbm{T}') = \nu(\la)$.
Otherwise, we restrict the action of $s_r$ to the linear span of $\mathbbm{T},\mathbbm{T}_{r \leftrightarrow r+1}$ and consider
its matrix w.r.t. the basis$(\mathbbm{T},\mathbbm{T}_{r \leftrightarrow r+1})$. It is
$$
s_r = \begin{pmatrix} m_r(\mathbbm{T}) & 1 + m_r(\mathbbm{T}_{r \lra r+1}) \\ 1 + m_r(\mathbbm{T}) & m_r(\mathbbm{T}_{r \lra r+1})
\end{pmatrix}
$$
and $\det(s_r) = - \alpha$ hence
$$
\ ^t s_r^{-1} = \frac{-1}{\alpha} \begin{pmatrix} m_r(\mathbbm{T}_{r \lra r+1} ) & - (1+m_r(\mathbbm{T})) \\ - (1+m_r(\mathbbm{T}_{r \lra r+1})) & m_r(\mathbbm{T})
\end{pmatrix}
$$
On the other hand, we have $\mathcal{L}^2 = \nu(\la) \mathrm{Id}$ hence
$\mathcal{L}^{-1} : \mathbbm{T} \mapsto w(\mathbbm{T}') \mathbbm{T}'$ and, since $w(\mathbbm{T}'_{r \lra r+1}) = - w(\mathbbm{T}')$, we
get $\mathcal{L} s_r \mathcal{L}^{-1}. \mathbbm{T} = \nu(\la) m_r(\mathbbm{T}') \mathbbm{T} - \nu(\la) (1+ m_r(\mathbbm{T}')) \mathbbm{T}_{r \lra r+1}$. 
It follows that
$$
\frac{\mathcal{L} s_r \mathcal{L}^{-1}}{\nu(\la)} : \mathbbm{T} \mapsto m_r(\mathbbm{T}') \mathbbm{T} - (1+m_r(\mathbbm{T}')) \mathbbm{T}_{r \lra r+1}
= m_r(\mathbbm{T}_{r \lra r+1}) \mathbbm{T} - (1+m_r(\mathbbm{T}_{r \lra r+1})) \mathbbm{T}_{r \lra r+1}
$$
which proves the formula. 

\end{proof}

As a consequence, we get
\begin{prop} If $\la \neq \la'$, the restriction to $\mathcal{B}_n$ of $R_{\la} \times R_{\la'} : B_n \to \GL(V_{\la}) \times \GL(V_{\la'})$
factors through the restriction of $R_{\la}$ and $(Q \mapsto (Q, \mathcal{L}^{-1} \ ^t Q^{-1} \mathcal{L})$
\end{prop}

\begin{lemma} \label{lem:restrBBn} If the order of $\alpha$ is $>n$, and $n \geq 2$, then the following are true.
\begin{enumerate}
\item For all $\la \vdash n$, the restriction of $R_{\la}$ to $\mathcal{B}_n$ is absolutely
irreducible. 
\item Let $\la,\mu \vdash n$ such that $\dim V_{\la} , \dim V_{\mu} > 1$. If the restrictions of $R_{\la}$ and $R_{\mu}$ to
$\mathcal{B}_n$ are isomorphic, then $\la = \mu$.
\item Let $\la,\mu \vdash n$ such that $\dim V_{\la} , \dim V_{\mu} > 1$. If the restrictions of $R_{\la}$ and of the dual representation of $R_{\mu}$ to
$\mathcal{B}_n$ are isomorphic, then $\la = \mu'$.
\end{enumerate}
\end{lemma}
\begin{proof}
We prove (i) by induction on $n$, the cases $n \leq 5$ being a consequence
of \cite{BM}. Let $U$ be a $\mathcal{B}_n$-stable subspace of $V_{\la} \otimes_{\F_q} k$,
for some extension $k$ of $\F_q$. By the branching rule and the induction assumption,
the action of $B_{n-1}$ on $V_{\la}$ is semisimple, and the decomposition of $V_{\la}$
as a direct sum of simple modules for $B_{n-1}$ is also a decomposition in a
sum of simple modules for $\mathcal{B}_{n-1}$. From this it follows that every
simple submodule of $U$ for the action of $\mathcal{B}_{n-1}$ is also $B_{n-1}$-stable,
hence $U$, being semisimple, is also $B_{n-1}$-stable. Since $B_n$ is generated
by $B_{n-1}$ and $\mathcal{B}_n$ it follows that $U$ is $B_n$-stable hence $U = V_{\la}$ and this concludes the
proof of (i).

We now prove (ii). By lemma \ref{lem:gprime}
and because the abelianization of $B_n$ is given by $\ell : B_n \onto \Z$,
$\sigma_i \mapsto 1$, this means that $R_{\mu}(b) = R_{\la}(b)u^{\ell(b)}$
for some $u \in \F_q^{\times}$, and this for all $b \in B_{n-1}$. 
This implies that the spectrum of $R_{\mu}(s_1)$, which is $\{ -1, \alpha \}$, is
also equal to $\{ -u,u\alpha \}$. Since we assumed $\alpha^2 \neq 1$ this is possible
only if $u=1$ hence
$R_{\mu} = R_{\la}$, which is excluded because these
two representations of the Hecke algebras
are non-isomorphic by assumption. 

The proof of (iii) is similar, once we notice that the restriction to $\mathcal{B}_n$ of the dual representation of $R_{\mu}$
is isomorphic to the restriction of $R_{\la'}$, by the above results.

\end{proof}

We now let $\la = \la_r = [n-r,1^r]$ and we want to compare $R_{\la_r}$ with $\Lambda^r R_{\la_1}$.
Any standard tableau of shape $\la_r$ can be indexed by the set of
indices $I = \{ i_1,\dots,i_r \} \subset \{ 2, \dots, n \}$, assuming $i_1 < \dots < i_r$, where each $i_k$ is the content
of the unique box of the diagram in line $k+1$. We let
$v_I$ denote the corresponding standard tableau, and we let $v_i = v_{\{ i \}}$.
Note that, when $\{k , k+1 \} \not\subset I$, then 
$ct(v_I:k)/ct(v_I:k+1)$ only depends on the number of boxes lying between $k$ and $k+1$ inside the
hook-shaped tableau $v_I$, and therefore only on $k$.
From this
we get by explicit computation that,
if $k \in I$ but $k+1 \not\in I$, then
$$
m_k(v_I) = \frac{\alpha-1}{1- \alpha^{-k}},
1+m_k(v_I) = \frac{\alpha-\alpha^{-k}}{1- \alpha^{-k}},
$$
and,   if $k \not\in I$ but $k+1 \in I$, then
$$
m_k(v_I) = \frac{\alpha-1}{1- \alpha^{k}},
1+m_k(v_I) = \frac{\alpha-\alpha^{k}}{1- \alpha^{k}}.
$$
It follows that
\begin{itemize}
\item if $k,k+1 \not\in I$, $s_k . v_I = \alpha v_I$,
\item if $k,k+1 \in I$, $s_k . v_I = -v_I$,
\item if $k \in I$, $k+1 \not\in I$, $s_k.v_I = \frac{\alpha-1}{1-\alpha^{-k}} v_I + \frac{\alpha - \alpha^{-k}}{1 - \alpha^{-k}} v_{I \Delta \{k,k+1 \}}$
\item if $k \not\in I$, $k+1 \in I$, $s_k.v_I = \frac{\alpha-1}{1-\alpha^{k}} v_I + \frac{\alpha - \alpha^{k}}{1 - \alpha^{k}} v_{I \Delta \{k,k+1 \}}$
\end{itemize}
where $\Delta$ denotes the symmetric difference : $A \Delta B = (A \cup B) \setminus (A \cap B)$.

On the other hand, to such an $I = \{ i_1< \dots < i_r \}$ we can associate
$u_I = v_{i_1} \wedge \dots \wedge v_{i_r} \in \Lambda^r V_{\la_1}$. By direct computation we get
\begin{itemize}
\item if $k,k+1 \not\in I$, $s_k . u_I = \alpha^r u_I$,
\item if $k,k+1 \in I$, $s_k . u_I = -\alpha^{r-1}u_I$,
\item if $k \in I$, $k+1 \not\in I$, $s_k.u_I = \frac{\alpha-1}{1-\alpha^{-k}} \alpha^{r-1} u_I + \frac{\alpha - \alpha^{-k}}{1 - \alpha^{-k}} \alpha^{r-1}u_{I \Delta \{k,k+1 \}}$
\item if $k \not\in I$, $k+1 \in I$, $s_k.u_I = \frac{\alpha-1}{1-\alpha^{k}}\alpha^{r-1} u_I + \frac{\alpha - \alpha^{k}}{1 - \alpha^{k}} \alpha^{r-1}u_{I \Delta \{k,k+1 \}}$
\end{itemize}
meaning that, if we identify
these two vector spaces via $v_I \lra u_i$, we have
$(\Lambda^r R_{\la_1})(s_r) = \alpha^{r-1} R_{\la_r}(s_r)$. Therefore, we get
$(\Lambda^r R_{\la_1})(g) = \alpha^{(r-1)\ell(g)} R_{\la_r}(g)$ for all $g \in B_n$,
and the following

\begin{prop} For every $r \in \{1,n-1 \}$, the restriction to $\mathcal{B}_n$ of
the morphism $R_{[n-r,1^r]} : B_n \to \GL(V_{[n-r,1^r]})$ factors through
the restriction of $R_{[n-1,1]} : B_n \to \GL(V_{[n-1,1]})$ to $\mathcal{B}_n$.
\end{prop}

Now assume that $\F_q = \F_p(\alpha)$ but $\F_p(\alpha+\alpha^{-1}) \neq \F_q$.
In that case there exists an involutive field automorphism $\eps : x \mapsto \bar{x}$ of $\F_q$ defined
by $\bar{\alpha} = \alpha^{-1}$. We define a hermitian form on $\mathcal{V}$ by the
formula $\langle \mathbbm{T}_1, \mathbbm{T}_2 \rangle = d(\mathbbm{T}_1) \delta_{\mathbbm{T}_1,\mathbbm{T}_2}$
where
$$
d(\mathbbm{T}) = \prod_{\stackrel{i<j}{\mathfrak{r}(i)>\mathfrak{r}(j)}} \frac{\alpha^{\mathfrak{c}(j) - \mathfrak{r}(j)} - \alpha^{\mathfrak{c}(i) - \mathfrak{r}(i)+1}}{ \alpha^{\mathfrak{c}(j)- \mathfrak{r}(j)+1} - \alpha^{\mathfrak{c}(i)-\mathfrak{r}(i)}}
$$
where $\mathfrak{c}(k)$, respectively $\mathfrak{r}(k)$, denote the column, respectively row, of $\mathbbm{T}$ where $k$ lies.

\begin{prop} The action of $B_n$ on $\mathcal{V}$ is unitary with respect to the above hermitian form. The restriction of this hermitian
form on every subspace $V_{\la}$ is nondegenerate. 
\end{prop}
\begin{proof}
We need to check that $\langle s_r \mathbbm{T}_1,s_r \mathbbm{T}_2 \rangle = \langle  \mathbbm{T}_1, \mathbbm{T}_2 \rangle$
for all standard tableaux $\mathbbm{T}_1, \mathbbm{T}_2$. If $r$ lies in the same row or the same column of $\mathbbm{T}_1$
or $\mathbbm{T}_2$ then the equality simply follows from $\alpha \bar{\alpha} = (-1)^2 = 1$. If not, then we can assume
that $\mathbbm{T}_2$ is either $\mathbbm{T} = \mathbbm{T}_1$ or $\mathbbm{T}_{r\lra r+1}$, and thus we only need to
check that the action of $s_r$ on the plane spanned by $\mathbbm{T}, \mathbbm{T}_{r \lra r+1}$ is unitary
with respect to the induced hermitian form. We express $s_r$ in the basis $(\mathbbm{T},\mathbbm{T}_{r \lra r+1})$.
In order to check the unitarity, up to a harmless exchange of $\mathbbm{T}$ and $\mathbbm{T}_{r \lra r+1}$,
we can assume that $\mathfrak{r}_{\mathbbm{T}}(r) < \mathfrak{r}_{\mathbbm{T}}(r+1)$. Then we get
$$
d(\mathbbm{T}_{r \lra r+1} ) = d(\mathbbm{T}) \frac{\alpha^{\mathfrak{c}(r+1) - \mathfrak{r}(r+1)} - \alpha^{\mathfrak{c}(r) - \mathfrak{r}(r)+1}}{ \alpha^{\mathfrak{c}(r+1)- \mathfrak{r}(r+1)+1} - \alpha^{\mathfrak{c}(r)-\mathfrak{r}(r)}}
= d(\mathbbm{T}) \frac{\alpha^{v - u} - \alpha^{j -i+1}}{ \alpha^{v-u+1} - \alpha^{j-i}}
$$
where $(i,j)$ and $(u,v)$ are the coordinates of $r$ and $r+1$ inside $\mathbbm{T}$, respectively.
It remains to check that, if $D = \diag(1,  \frac{\alpha^{v - u} - \alpha^{j -i+1}}{ \alpha^{v-u+1} - \alpha^{j-i}})$
and $S_r$ is the $2 \times 2$ matrix representing the action of $s_r$, then
we have $D S_r = \, ^t \overline{S_r}^{-1} D$, and this is straightforward.
The nondegeneracy statement is obvious.

\end{proof}

In these circumstances, we have that 

\begin{lemma} We assume that the order of $\alpha$ is $>n$, 
$\F_p(\alpha+\alpha^{-1}) \neq \F_q$
and $\la,\mu \vdash n$ with $\dim V_{\la} > 1$. 
If $n \geq 3$,
then the restrictions to $\mathcal{B}_n$ of $R_{\la}$ and $\overline{R_{\mu}}^*$ are isomorphic
iff $\la =\mu$.
\end{lemma}
\begin{proof}
Because of the unitary structure we get that the restrictions to $\mathcal{B}_n$ of $R_{\la}$ and of
its conjugate-dual $\overline{R_{\la}}^*$ are isomorphic. Under the assumptions of the
lemma this means that the restrictions of $R_{\la}$ and $R_{\mu}$ are isomorphic, and this
implies $\la = \mu$ by lemma \ref{lem:restrBBn}. 

\end{proof}

\section{Representation-theoretic technicalities}

We also need to consider the set of elements that preserve both a unitary and an orthogonal/symplectic
form. If $\varphi$ denotes a nondegenerate bilinear form over $\F_q^N$ we let $OSP_N(\varphi)$
denotes the group of isometries of this form ; if $\psi$ is an hermitian form, we let $U_N(\psi)$ denote its group
of isometries. We will use the following property, which is probably folklore.

\begin{prop} Let $q = u^2$, $\varphi$ a nondegenerate bilinear form over $\F_q^N$, $\psi$ a nondegenerate hermitian
form over $\F_q^N$.
If $G \subset
OSP_N(\varphi) \cap U_N(\psi)$ is absolutely irreducible, then there exists $x \in \GL_N(q)$ and a nondegenerate bilinear
form $\varphi'$ over $\F_u^N$ such that $^x G \subset OSP(\varphi')$. Moreover, $\varphi'$ is (skew-)symmetric if and only if $\varphi$ is so.
\end{prop}
\begin{proof} We let $R : G \to \GL_N(q)$ denote the natural inclusion, and we consider it as a linear
representation of $G$. We set $\Gamma = \Gal(\F_q/\F_u) = \{ \Id, \eps \}$ and use both notations $\eps(x) = \bar{x}$.
We have $R^* \simeq R$ and $\overline{R}^* \simeq R$ hence $R \simeq \overline{R}$. As a consequence there exists 
$P \in \GL_N(q)$ such that $\overline{R}(g) = P R(g) P^{-1}$ for all $g \in G$. It follows that
$\overline{P}P$ commutes to every $\overline{R}(g)$. By Schur's lemma and the absolute irreducibility of $G$
we get $\overline{P}P \in( \F_q^{\times})^{\Gamma} = \F_u^{\times}$. Since the norm map $\F_q^{\times} \to \F_u^{\times}$
is surjective, we have $\overline{P}P = \overline{\la} \la$ for some $\la \in \F_q^{\times}$ and thus, replacing
if needed $P$ with $P \la^{-1}$, we may assume $\overline{P} P = \Id$. Then $\Id \mapsto \Id$, $\eps \mapsto P$
defines an element in $Z^1(\Gamma, \GL_N(q))$. By Hilbert's theorem 90 it follows that there
exists $S \in \GL_N(q)$ such that $P = \overline{S} S^{-1}$. Then, setting $R'(g) = S^{-1} R(g) S$, we have $\overline{R}'(g)
\in \GL_N(\F_u)$. Moreover, $R'(g)$ preserves the bilinear form deduced from $\varphi$ : in matrix form, if
$W$ denotes the matrix of $\varphi$ in the canonical basis of $\F_q^N$, we have $^t R(g) W R(g) = W$
for all $g \in G$, hence $R'(g)$ preserves the bilinear form $\varphi^S$ given by the matrix $W^S = ^t S W S \in \GL_N(q)$.
Since $R'(g) \in \GL_N(u)$ it also preserves all the $\tilde{W}_{\la} = \la W^S + \overline{\la} \overline{W}^S$ for all $\la \in \F_q^{\times}$.
Since $W^S \neq 0$, there exists $\la \in \F_q^{\times}$ such that  $\tilde{W}_{\la} \neq 0$, for otherwise
$\la/\overline{\la} = \mu/\overline{\mu}$ for all $\la,\mu \in \F_q^{\times}$, and this would imply $u=q$. Then $\tilde{W}_{\la}$
for such a $\la$ defines a bilinear form $\varphi'$ over $\F_u^N$, and we have $R'(g) \in \OSP(\varphi')$
for all $g \in G$, hence $^x G \subset \OSP(\varphi')$ for $x = S^{-1}$. The last part of the statement is a consequence of
our construction of $\varphi'$.
\end{proof}

\section{Proof of the main theorem}

We let $\mathcal{E}_n$ denote the set of partitions on $n$ which are not hooks. From 
section \ref{sectmainfactor}
we know that the morphism $\mathcal{B}_n \to H_n(\alpha)^{\times} \simeq \prod_{\la \vdash n} \GL(\la)$ factorizes though the morphism
$$
\Phi_n : \mathcal{B}_n  \to \SL_{n-1}(q) \times \prod_{\stackrel{\la \in \mathcal{E}_n}{\la < \la' }} \SL(\la) \times \prod_{\stackrel{\la \in \mathcal{E}_n}{\la = \la' }} \OSP'(\la) 
$$
where $\OSP'(\la) = G(\la)$ denotes the commutator subgroup of the group of isometries of
the bilinear form defined in section \ref{sectmainfactor}. In particular,
when $\la = \la'$, $p \neq 2$ and when the action of the braid group on $V_{\la}$ preserves an orthogonal
form, then $\OSP'(\la)$ denotes the group classically denoted $\Omega_N^+(q)$ (see \cite{WILSON}), where $N = \dim V_{\la}$.
We assume that $\F_p(\alpha+\alpha^{-1}) = \F_q = \F_p(\alpha)$ and, as in \cite{BM}, that the order of $\alpha \in \F_q^{\times}$ is not
$2,3,4,5,6,10$.
Theorem \ref{maintheo} in that case states that $\Phi_n$ is surjective when the order of $\alpha$ is in addition greater than $n$. For $n \leq 5$ this is a consequence of \cite{BM}. We then proceed by induction on
$n$, assuming that $\Phi_{n-1}$ is surjective and $n \geq 6$. We first prove that each of the composites $R_{\la}$ of $\Phi_n$ with the projection
on the quasi-simple factor attached to $\la$ is surjective. For this, let $\la \in \mathcal{E}_n$. If $\la$ has at most two rows or
at most two columns this is a consequence of \cite{BM}, so we can assume that $\la$ contains $[3,2,1]$, hence $\dim V_{\la} \geq 16$. 
Moreover, for $n = 6$ the only case to be taken care of is $\la = [3,2,1]$.
Finally note that, since $n \geq 6$, our
assumptions imply that $\alpha$ has order at least $7$, hence $q \geq 8$.

We use the notation $\mu \subset \la$ to indicate the inclusion of the corresponding Young
diagrams, namely that $\mu_i \leq \la_i$ for all $i$'s. By the induction assumption, we know that 
\begin{itemize}
\item if $\la \neq \la'$, there exists $\mu \subset \la$ of size
 $n-1$ such that $\mu' \not\subset \la$ and such that $\mu \supset [3,2]$ or $\mu \supset [2,2,1]$ (this is because $\la$ is equal to the union of the $\mu$'s of size $n-1$ contained in it such that $\mu \supset [3,2]$ or $\mu \supset [2,2,1]$). In particular
 $\mu \neq \mu'$. Since $\mu$ is not a hook, by the induction assumption it follows that the image of $\mathcal{B}_{n-1}$ contains a direct factor
 $\SL(\mu)$ and in particular some $\SL_2(q)$ acting naturally on some $2$-dimensional subspace
and some $\SL_3(q)$ acting naturally on some 3-dimensional subspace. 
\item if $\la = \la'$ and there exists $[3,2] \subset \mu \subset \la$ of size $n-1$ with $\mu \neq \mu'$, then $\mu' \subset \la$.
By the induction assumption the image of $\mathcal{B}_{n-1}$ contains a subgroup acting on a subspace of dimension $2 \dim V_{\mu}$ as
$\{ x \oplus ^t x^{-1} \ | \ x \in \SL_{\dim V_{\mu}}(q) \}$. Since $\dim V_{\mu} \geq 3$ it contains in particular a subgroup acting on a subspace of dimension $4$ as
$\{ x \oplus ^t x^{-1} \ | \ x \in \SL_2(q) \}$, and a subgroup acting on some $6$-dimensional subspace as
$\{ x \oplus ^t x^{-1} \ | \ x \in \SL_3(q) \}$.
\item if $\la = \la'$ and there does \emph{not} exists $[3,2] \subset \mu \subset \la$ of size $n-1$ with $\mu \neq \mu'$. In this case
it is easily checked that $\la$ is a square diagram, hence the restriction of $\la$ to $\mathfrak{S}_{n-1}$ is irreducible, and that the corresponding diagram $\mu$ satisfies $\mu = \mu'$, $\mu \supset  [3,2,1]$. 
Since the restriction to $\mathfrak{S}_{n-1}$ is irreducible one can check that $\OSP(\mu) = \OSP(\la)$ hence,
since $G \subset \OSP'(\la)$, we get $G = \OSP'(\la)$ and this case does not need to be considered further. 

\end{itemize}

We notice that $\{ x \oplus ^t x^{-1} \ | \ x \in \SL_2(q) \}$
contains the element
$$
\begin{pmatrix} 1 & 1 & 0 & 0 \\ 0 & 1 & 0 & 0 \\ 0 & 0 & 1 & 0 \\ 0 & 0 & -1 & 1 \end{pmatrix}
$$
hence $R_{\la}(\mathcal{B}_n)$ contains in all cases an element $x$ such that
$[x,V_{\la}] = (x-1) V_{\la}$ has dimension $2$, this being obvious when it contains
a natural $\SL_2(q)$. We then use the
following result of \cite{GURALSAXL}, for $V = V_{\la}$, $G = R_{\la}(\mathcal{B}_n)$.

\begin{theor} \label{theoGURALSAXL} (\cite{GURALSAXL}, theorem 7.A) 
Let $V$ be a finite dimensional vector space of dimension $d > 8$ over
an algebraically closed field $\F$ of characteristic $p >0$. Let $G$ be a finite irreducible
subgroup of GL(V ) which is primitive and tensor-indecomposable on V . Define $\nu_G(V )$
to be the minimum dimension of $[\beta g, V ] =(\beta g- 1)V$ for $g \in G$,
$\beta$ a scalar with $\beta g \neq1$. Then either $\nu_G(V ) > max(2, \sqrt{d}/2)$
or one of the following holds:
\begin{enumerate}
\item $G$ is classical in a natural representation
\item  $G$ is alternating or symmetric of degree $c$ and $V$ is the deleted permutation module of
dimension $c-1$ or $c-2$.
\item $F^*(G)=U_5(2)$ with $p\neq2, d=10$.
\end{enumerate}
\end{theor}
Note that $(iii)$ does not occur because $d \geq 16$. 
If $G$ contains a natural $\SL_2(q)$,
then
$G$ is tensor-indecomposable by the following lemma. 

\begin{lemma} \label{tensindec1} If $d \geq 6$ and $G \subset \GL_d(\F)$ contains an element conjugated to $\diag(\zeta,\zeta^{-1},1,1,\dots)$
for $\zeta^2 \neq 1$ 
, then $G$ is tensor-indecomposable.
\end{lemma}
\begin{proof}
Let $r$ denote the order of $g$. Since $\F$ has characteristic $p$, we know that $r$
is coprime to $p$. Assume by contradiction that $G$ is tensor-decomposable. Then, $g$
could be written $g_1 \otimes g_2$, hence $g^r = 1$ implies 
that $g_1^r = t$ and $g_2^r = t^{-1}$ for some $t \in \F^{\times}$. Since $r$ is prime to $p$, $X^r - t^{\pm 1}$
has no multiple root and thus $g_1,g_2$ are semisimple.

Assume $d=ab$ with $g_1 \in \GL_a(\F)$ and $g_2 \in \GL_b(\F)$ and let $\la_1,\dots,\la_a$ and $\mu_1,\dots, \mu_b$ their eigenvalues.
We can assume $a \geq 3$,$ b\geq 2$,
and $\la_1 \mu_1 = \zeta$. We let $\la_1 = \beta$, hence $\mu_1 = \beta^{-1} \zeta$.
Up to reordering, there are only three cases :
\begin{itemize}
\item either $\la_2 \mu_2 = \zeta^{-1}$, and then $\la_1 \mu_2 = 1$ hence $\mu_2 = \beta^{-1}$;
\item or  $\la_1 \mu_2 = \zeta^{-1}$, and then $\mu_2 = \beta^{-1} \zeta^{-1}$;
\item or $\la_2 \mu_1 = \zeta^{-1}$, that is $\la_2 = \beta^{-1}$, and then $\la_1 \mu_2 = \la_2 \mu_2 = 1$ implying $\la_2 = \beta$ and $\beta^2 = 1$, 
hence also $\mu_2 = \beta^{-1} = \beta$, $\mu_1 = \beta \zeta$.
\end{itemize}
In these three cases, the fact that $1 = \la_3 \mu_1 = \la_3 \mu_2 \Rightarrow \mu_1 = \mu_2$ yields a contradiction.
\end{proof}

If $G$ does not contain a natural $\SL_2(q)$, then it contains a twisted-diagonal embedding of $\SL_2(q)$ and therefore an element which is conjugated to $\diag(\zeta,\zeta,\zeta^{-1},\zeta^{-1},1,\dots,1)$ with $\zeta$ of
order $q-1$. It is therefore tensor-indecomposable by the following lemma.

\begin{lemma} \label{tensindec2} If $d \geq 16$ and $G \subset \GL_d(\F)$ contains an element of order prime to $p$
and conjugated to $\diag(\zeta,\zeta,\zeta^{-1},\zeta^{-1},1,\dots,1)$ with $\zeta^2 \neq 1$,
then $G$ is tensor indecomposable, except possibly if $G = G_1 \otimes G_2$ with $G_1 \subset \GL_2(\F)$.
\end{lemma}
\begin{proof} 
We let $g$ denote the element of the statement, and assume by contradiction that $g = g_1 \otimes g_2$ with $g_1 \in \GL_a(\F)$,
$g_2 \in \GL_b(\F)$, $ab = d$ and $a,b \geq 3$. Since $d \geq 16$ we can assume $a \geq 3$, $b \geq \sqrt{d} \geq 4$.
As in the proof of the previous lemma, the order condition imply that $g_1$ and $g_2$ are semisimple.
Let $\la_1,\la_2,\dots$ and $\mu_1,\mu_2,\dots$ denote the eigenvalues of $g_1$ and $g_2$, respectively.
Up to reordering we can assume $\la_1 \mu_1 = \zeta$.

Let us first assume there exists $i \neq 1$ such that $\la_1 \mu_i = \zeta$. Up to reordering we can assume $i = 2$, hence
$\la_1 \mu_2 = \la_1 \mu_1 = \zeta$, hence $\mu_1 = \mu_2$. Hence $\la_2 \mu_1 = \la_2 \mu_2 \in \{ \zeta^{-1},1 \}$.
If $\la_2 \mu_1 = \la_2 \mu_1 = \zeta^{-1}$, we then have  $\la_1 \mu_3 = \la_2 \mu_3 = 1$, and thererefore $\la_1 = \la_2$.
But then $\zeta = \la_1 \mu_1 = \la_2 \mu_1 = \zeta^{-1}$, contradicting $\zeta^2 \neq 1$.

On the other hand, if $\la_2 \mu_1 = \la_2 \mu_2 = 1$, when $b \geq 5$ there exists $i>2$ such that $\la_2 \mu_i = 1$. But then $\mu_i = \mu_2$,
hence $\la_1 \mu_i = \la_1 \mu_2 = \zeta$ and $\zeta$ would appear with multiplicity $3$, a contradiction.
If $b = 4$ and there is no $i>2$ such that $\la_2 \mu_i = 1$, then $\la_2 \mu_3 = \la_2 \mu_4 = \zeta^{-1}$. Since 
$a \geq 3$ this implies $\la_3 \mu_2 = 1 = \la_3 \mu_3$ hence $\mu_2 = \mu_3$, contradicting $\zeta = \la_1 \mu_2 = \la_1 \mu_3 = 1$.

We can thus assume without loss of generality that, for all $i \neq 1$, we have $\la_1 \mu_i \neq \zeta$.

Let us assume now that $\la_2 \mu_1 = \zeta$. Since $\la_1 \mu_1 = \zeta$ we have $\la_2 = \la_1$. Since $a \geq 3$
we get $\la_3 \mu_i = 1$ for all $i$, hence $\mu_1 = \mu_2 = \mu_3$ and $\zeta = \la_1 \mu_1 = \la_1 \mu_2$, contradicting
$\la_1 \mu_2 \neq \zeta$.

We can thus now assume without loss of generality that, for all $i \neq 1$, we have $\la_1 \mu_i \neq \zeta$ and $\la_2 \mu_1 \neq \zeta$.
Up to reordering we can thus moreover assume $\la_2 \mu_2 = \zeta$. If there existed $i>2$ such that $\la_1 \mu_i = \la_2 \mu_i$, 
then the consequence $\la_1 = \la_2$ would contradict $\la_1 \mu_2 \neq \la_2 \mu_2$. We can thus assume that
\begin{itemize}
\item either $\la_1 \mu_3 = \la_1 \mu_4 = 1$, $\la_2 \mu_3 = \la_2 \mu_4 = \zeta^{-1}$, but then
$\la_1 \mu_2 = 1 = \la_1 \mu_3$ hence $\mu_2 = \mu_3$ and $\la_2 \mu_3 = \zeta^{-1} = \la_2 \mu_2 = \zeta$,
a contradiction;
\item or $\la_1 \mu_4 = \zeta^{-1} = \la_2 \mu_3$, but then $\la_1 \mu_2 = \la_1 \mu_3 = 1$ hence $\mu_2 = \mu_3$ and
$\zeta = \la_2 \mu_2 = \la_2 \mu_3 = \zeta^{-1}$, a contradiction;
\item or $\la_1 \mu_2 = \zeta^{-1}$ or $\la_2 \mu_1  =\zeta^{-1}$, in which case there would exist $i \in \{ 3,4 \}$ such that
$\la_1 \mu_i = \la_2 \mu_i$, hence $\la_1 = \la_2$, contradicting $\la_1 \mu_2 \neq \la_2 \mu_2$.
\end{itemize}
This concludes the proof.
\end{proof}

We now want to rule out 
case (ii) of theorem \ref{theoGURALSAXL}. For this, we first consider the case where $G$ contains
a natural $\SL_2(q)$.
In particular, it contains an element $g$ of order $q-1$ such that $\dim [g,V] = 2$. 
In case $G \subset \mathfrak{S}_m$ and $V$ is the deleted representation of $\mathfrak{S}_m$
of dimension $N =m-1$ or $N = m-2$ we notice that, the order of $g$ being coprime to $p$,
it acts as a semisimple endomorphism on the permutation representation $\tilde{V}$ of $\mathfrak{S}_m$ ; since the composition
factors of $\tilde{V}$ are $V$ together with one or two copies of the trivial module, we get that $\dim [g,\tilde{V}] = \dim [g,V]$. But the
condition $\dim [g,V] \leq 2$ implies that $g \in \mathfrak{S}_m$ has order at most $3$, a contradiction since $q \geq 8$.
The other case is when $G$ contains a twisted-diagonal embedding of $\SL_2(q)$.
 In this case it contains an
element $g$ conjugated to $\diag(\zeta,\zeta,\zeta^{-1},\zeta^{-1},1 ,\dots,1)$ of order $q-1 \geq 7$. We similarly get that,
since $\dim [g,V] \leq 4$, the order of $g$ can be at most $6$, a contradiction.

Next we want to show that the action of $G$ on $V$ is primitive. We start by ruling out
the monomial case. If $G \subset \F_q^{\times} \wr \mathfrak{S}_N$ then 
we use the fact $\SL_2(q)$ has a $p$-Sylow of order $q$,
all of whose elements $h$ satisfy $\dim [h,\F_{q^2}] \leq 1$, and therefore
$G$ contains an elementary abelian $p$-subgroup of order $q$,
whose elements $h$ satisfy $\dim [h,V] \leq 2$. By
the Sylow theory these $p$-subgroups are conjugated inside $ \F_q^{\times} \wr \mathfrak{S}_N \subset \GL(V)$ to a $p$-subgroup of $\mathfrak{S}_N$, since the order of $(\F_q^{\times})^N$ is coprime to $p$. This means that $\mathfrak{S}_N$
contains an elementary abelian $p$-subgroup $H$ of order $q$ such that, for all $h \in H \ \dim [h,V] \leq 2$.

We then use the following lemma.

\begin{lemma} \label{lemGelmab}
Let $G$ be an elementary abelian $p$-subgroup of $\mathfrak{S}_N$ of order $p^r$. Then $G$ contains an
element which is a product of at least $r$ disjoint $p$-cycles.
\end{lemma}
\begin{proof}
By the permutation action we can identify $\mathfrak{S}_N$ and thus $G$ with a subgroup of $\GL_N(\C)$.
Since $G$ is commutative, it is conjugated to a group of diagonal matrices, and therefore
can be identified with a subgroup of $\mu_p^N$, where $\mu_p$ denotes the group of $p$-th roots of $1$ in $\C$.
Let $\zeta \in \mu_p$ be a primitive $p$-th root of $1$. Every $g \in G$ is a product of $m$
cycles, with $m$ equal to the multiplicity of $\zeta$ in the spectrum of $g$. We thus need to prove
that there exists $g \in G \subset \mu_p^N$ having at least $r$ components equal to $\zeta$.

Identifying $\mu_p$ with $\F_p$ such that $\zeta \mapsto 1$, we get a structure of $\F_p$-vector space
on $\mu_p^N$, and the lemma follows from the following one.
\begin{lemma}
Let $K$ be a field, $V$ a $r$-dimensional subspace of $K^N$. There exists $v \in V$ having at least
$r$ entries equal to $1$.
\end{lemma}
\begin{proof}
Let $e_1^*,\dots,e_N^*$ denote the dual canonical basis of $K^N$, and $J \subset \{1,\dots,N\}$ of maximal
cardinality containing an element $v$ with $e_i^*(v) = 1$ for all $i \in J$. If $|J|<r$,
the intersection of the hyperplanes $\Ker (e_i^*)$ for $i \in J$ and $V$ would contain a non-zero
element $w$. Moreover, we have an element $v \in J$ such that $e_i^*(v) = 1$ for all $i \in J$.
Then $e_i^*(v + \beta w) = 1$ for all $\beta \in K$ and $i \in J$. Since $w \neq 0$ there exists $i_0 \not\in J$
such that $e_{i_0}^*(w) \neq 0$. Therefore, we can find $\beta$ such that $e_{i_0}(v + \beta w) =1$, and this
contradicts the maximality of $J$.
\end{proof}\end{proof}

By lemma \ref{lemGelmab} the group $H$ contains a product $h$ of $r$ disjoint $p$-cycles. Since $\dim [h,V] = (p-1)r$
we get $(p-1)r \leq 2$, contradicting assumption $q > 4$.

We now want to rule out the non-monomial imprimitive case. Assume by contradiction that $G \subset H \wr \mathfrak{S}_m = (H_1 \times \dots H_m)\rtimes \mathfrak{S}_m$, where $H_1,\dots,H_m$ denote the $m$ copies of $H \simeq \SL_{N/m}(q)$ which are permuted by the action of $\mathfrak{S}_m$.
Let us consider an element $t$ of order $p$ which is either a transvection or an element of Jordan form $M \oplus \mathrm{Id}_{N-4}$ with
$$
M = \begin{pmatrix} 1 & 1 & 0 & 0 \\ 0 & 1 & 0 & 0 \\ 0 & 0 & 1 & 1 \\ 0 & 0 & 0 & 1 \end{pmatrix}
$$
Notice that in both cases $G$ contains such an element. The rank of $t-1$ is at most $2$. Assume also that $t \not\in H_1 \times \dots \times H_m$. Up to reordering we can assume $H_1^t \neq H_1$. Since $t$ has order $p$ we can thus assume
$H_i^t = H_{i+1}$ for $1 \leq i \leq p-1$, $H_{p-1}^t = H_p$. We let $U_1 \oplus \dots \oplus U_m$ be the direct sum decomposition
corresponding to the wreath product. Let $v_1 \in U_1 \setminus \{ 0 \}$. By completing the
family $(v_1,tv_1,\dots,t^{p-1} v_1)$ we get a basis on which $t$ acts by a matrix of the form $M_p  \oplus X$
where $M_p$ is the circulating matrix of order $p$ and $X$ is some matrix of size $N-p$. We have $(M-1)^2 = 0$
but $(M_p-1)^2 \neq 0$ whenever $p \geq 3$. Assuming this, we get $t \in H_1 \times \dots \times H_m$.
Notice that the induction assumption implies that $R_{\la}(\mathcal{B}_{n-1})$
is a direct product of quasi-simple groups containing elements of that type.
Because these elements are not semisimple, they moreover do not belong to the centers of these groups.
It follows that $R_{\la}(\mathcal{B}_{n-1})$
is normally generated by these elements
hence is included in $H_1 \times \dots \times H_{m}$, which is
normal in $H \wr \mathfrak{S}_m$. Since $\mathcal{B}_n$ is normally generated by $\mathcal{B}_{n-1}$
(see lemma \ref{lemnormgen})
this proves that $R_{\la}(\mathcal{B}_n) \subset H_1\times \dots \times H_m$, contradicting
the irreducibility of $R_{\la}$.

It then remains to examine separately the case $p=2$. If $\dim U_1 \geq 3$, we can pick a linearly
independent family $v_1,v'_1,v''_1 \in U_1$ and, by completing the family
$(v_1,tv_1,v'_1,tv'_1,v''_1,tv''_1)$ we get a basis on
which $t$ acts by a matrix of the form $M_p \oplus M_p \oplus M_p \oplus X$ for some $X$ and we get that the
rank of $t-1$ is at least $3$, a contradiction that proves $\dim U_1 \leq 2$. In case $t$ is a transvection, the same
contradiction proves $\dim U_1 = 1$, and we are reduced to the monomial case that we already did.
If we cannot choose $t$ to be a transvection, we have $p=2$, $\dim U_1 = 2$. Under our assumption we know $q \neq 2$. Let
us consider two $\F_2$-linearly independent elements $a_1,a_2 \in \F_q$, and elements $t_1,t_2 \in G$ whose Jordan form
in some common basis is
$$
t_i = \begin{pmatrix} 1 & a_i & 0 & 0 \\ 0 & 1 & 0 & 0 \\ 0 & 0 & 1 & a_i \\ 0 & 0 & 0 & 1 \end{pmatrix}
$$
Let us assume $t_1,t_2 \not\in H_1 \times \dots \times H_m$. By the same argument as above with $t=t_1$, we can
assume that $U = U_1 \oplus U_2$ is $t_1$-stable, with $t_1(U_1) = U_2$ and therefore $t_1(U_2) = t_1^2(U_1) = U_1$, and $t_1(U_i) = U_i$ for $i\geq 3$.
Using the same argument for $t_2$ we can also assume that $U' = U_a \oplus U_b$ is $t_2$-stable with $t_2$ exchanging $U_a$ and  $U_b$
for some $a \neq b$. 
Since $I = \Imm(t_i -1)$ is independent of $i$, we have $I \subset U \cap U'$. We prove that $U_r \not\subset I$ for every $r$. When
$r \not\in \{1,2,a,b \}$ this is clear because $t_i$ acts as 1 on such a $U_r$. But $U_r \subset I = \Imm(t_i-1) \subset \Ker(t_i -1)$ for all $i$ implies $r \not\in \{1,2,a,b \}$ since each of the
$U_r$ for $r \in \{1,2,a,b \}$ is not stable by at least one of the two $t_i$'s. 

Then, $U \cap U'$ containing
the 2-dimensional subspace $I$ but no $U_r$, we have $U =U'$.
It follows that $t_1t_2(U_r) = U_r$ for all $r$, hence $t =t_1 t_2 \in H_1 \times \dots \times H_m$.
We can thus resume the previous argument : since $R_{\la}(\mathcal{B}_{n-1})$ is normally generated by
such elements, and because $\mathcal{B}_n$ is normally generated by
$\mathcal{B}_{n-1}$, we would get $R_{\la}(\mathcal{B}_n) \subset H_1 \times \dots \times H_m$, contradicting
the irreducibility of $R_{\la}$.

This proves that $G$ is primitive, tensor-indecomposable, and we ruled out cases (ii) and (iii) of the theorem.

Theorem \ref{theoGURALSAXL} implies that $G$ is a classical group over a finite subfield $\F_{q'}$ of $\F_q$. We first show
that $q' = q$. We use the following lemmas, where $\SU_m(q)$ denotes, in case $q$ is a square, the
unitary subgroup of $\SL_m(q)$.

\begin{lemma} \label{lemmetraces}
For all $m \geq 2$, the field generated over $\F_p$ by $\{ \tr(g) ; g \in \SL_m(q) \}$ is $\F_q$.  
For all $m \geq 3$, the field generated over $\F_p$ by $\{ \tr(g) ; g \in \SU_m(q) \}$ is $\F_q$.
\end{lemma}
\begin{proof}
We start we the case $\SL_m(q)$ and argue by contradiction.
Suppose that $\{ \tr(g) ; g \in \SL_m(q) \}$
generates a proper subfield $\F_{q'}$ with $q' \leq \sqrt{q}$. Since the action of
$\SL_m(q)$ on its natural representation is absolutely irreducible, it would be conjugate
inside $\GL_m(q)$ to some subgroup of $\GL_m(q')$ (see e.g. \cite{ISAACS}, theorem 9.14), and therefore
to some subgroup of $\SL_m(q')$ since $\SL_m(q)$ is perfect. But $|\SL_m(q)| > |\SL_m(q')|$ for $m \geq 2$, a contradiction.
In the $\SU_m(q)$ case, $\{ \tr(g) ; g \in \SU_m(q) \}$ 
would generate a proper subfield $\F_{q'}$ with $q' \leq \sqrt{q}$. Since the action of
$\SU_m(q)$ on its natural representation is again absolutely irreducible, it would be conjugated
inside $\GL_m(q)$ to some subgroup of $\GL_m(q')$ by the same argument, and therefore
to some subgroup of $\SL_m(q')$ since $\SL_m(q)$ is perfect.
Then the order of $\SU_m(q)$ is at most
$$
|\SL_m(q')| = (q')^{\frac{n(n-1)}{2}} ((q')^2-1) \dots ((q')^n -1)
\leq \sqrt{q}^{\frac{n(n-1)}{2}} (\sqrt{q}^2-1) \dots (\sqrt{q}^n -1) = |\SL_n(\sqrt{q})|
$$
but $|\SU_m(q)| = \sqrt{q}^{\frac{m(m-1)}{2}} (\sqrt{q}^2-1) (\sqrt{q}^3+1) \dots  (\sqrt{q}^m-(-1)^m)$
and thus $|\SU_m(q)|>|\SL_m(\sqrt{q})|$ as soon as $m\geq 3$, a contradiction.

\end{proof}

Note that a similar statement does not hold for $\SU_2(q)$, for in that case every element of
the group has a trace of the form $\zeta + \bar{\zeta} \in \F_{\sqrt{q}}$.

If $G$ contains a natural $\SL_2(q)$, this proves $q'=q$. Otherwise, we can consider a twisted-diagonal embedding
of $\SL_3(q)$ as a representation $\rho : \SL_3(q) \to G \subset \GL_N(\F_{q'}) \subset \GL_N(\overline{\F_p})$, and assume
by contradiction that $\F_{q'}$ is a proper subfield of $\F_q$. Let $\varphi : \SL_3(\F_q) \to \GL_3(\overline{\F_p})$
denote the (absolutely irreducible) natural representation, and $\un$ the trivial one. We have $\rho \simeq \varphi \oplus \varphi^* \oplus
\un^{N-6}$. Let $\sigma$ be a generator of $\Gal(\F_q/\F_{q'})$. We have $\rho \simeq \rho^{\sigma}$
and therefore either $\varphi \simeq \varphi^{\sigma}$ or  $\varphi^* \simeq \varphi^{\sigma}$.
In the first case, Lemma \ref{lemmetraces} would imply $\F_q = (\F_q)^{\Gal(\F_q/\F_{q'})} = \F_{q'}$,
a contradiction. In the second case, we could have $\varphi^{\sigma^2} \simeq (\varphi^*)^{\sigma} \simeq (\varphi^{\sigma})^{\sigma} \simeq \varphi^{**} \simeq \varphi$ hence $\F_q = (\F_q)^{\langle \sigma^2 \rangle}$ by the same
argument. Thus $\sigma : x \mapsto \bar{x}$ has order $2$, $q$ is a square, and we have 
$\varphi^* \simeq \overline{\varphi}$. But this implies that $\varphi$ preserves some
hermitian nondegenerate form on $\F_q^3$, and therefore $\varphi$ would embed $\SL_3(q)$ into
some conjugate of $\SU_3(q)$, contradicting $|\SL_3(q)| >|\SU_3(q)|$ (see the proof of
Lemma \ref{lemmetraces}). Altogether this proves $q' = q$.

If $\la = \la'$ then we know $G \subset \OSP(\la)$, hence the only possibility left
by Theorem \ref{theoGURALSAXL} is that $G = \OSP'(\la)$. If $\la \neq \la'$, then $G$ cannot preserve any nontrivial bilinear form, since
$R_{\la}$, viewed as a representation of $\mathcal{B}_n$, is not isomorphic to its dual by Lemma \ref{lem:restrBBn}, and neither can it preserve an hermitian form, because it is also not
isomorphic to its conjugate-dual. This last property is because the restriction to $\mathcal{B}_3$ does not have this
property when $\F_p(\alpha+\alpha^{-1}) = \F_p(\alpha)$, as is shown in \cite{BM}.
Theorem \ref{theoGURALSAXL} thus implies $G = \SL(\la)$.

Now, we now recall Goursat's lemma,
which describes the subgroups of a direct
product, and that we need in the sequel.
\begin{lemma}(Goursat's lemma)
Let $G_1$ and $G_2$ be two groups, $H\leq G_1\times G_2$, and denote by
$\pi_i:H\rightarrow G_i$. Write $H_i=\pi_i(H)$ and $H^i=\ker(\pi_{i'})$,
where $\{ i, i' \} = \{ 1, 2 \}$.
Then there is an isomorphism  $\varphi:H_1/H^1\rightarrow H_2/H^2$ 
such that 
\begin{equation}
\label{eq:goursat}
H=\{(h_1,h_2)\in H_1\times H_2
\ |\ \varphi(h_1H^1)=h_2H^2\}.
\end{equation}
\label{lem:goursat}
\end{lemma}

We now can prove that $\Phi_n$ is surjective.
We choose a good ordering on the elements of $\mathcal{E}_n$ such that $\la \leq \la'$, with the
additional condition that the 2-rows diagram are smaller than the others.
By numbering the partitions $\la \in \mathcal{E}_n$ such that $\la \leq \la'$
we can prove by induction on $n$ that, for a given $\la_0$, the composite of
$\Phi_n$ with the projection of its target domain onto
$$
G_{\la_0} = \SL_{n-1}(q) \times \prod_{\stackrel{\la \in \mathcal{E}_n}{\la < \la', \la<\la_0 }} \SL(\la) \times \prod_{\stackrel{\la \in \mathcal{E}_n}{\la = \la', \la < \la_0 }} \OSP'(\la) 
$$
is surjective. For $\la_0$ the minimal element of $\mathcal{E}_n$, $G_{\la_0} = \SL_{n-1}(q)$.
By the results of \cite{BM} this composite is surjective whenever $\la_0$ is a 2-rows diagram.
We use Goursat's lemma with $G_1 = G_{\la_0}$ and $G_2 = G(\la_0+1)$,
where we let as in the introduction $G(\mu) = \SL(\mu)$ if $\mu \neq \mu'$, and $G(\mu) = \OSP'(\mu)$ otherwise.
We let $PG(\mu)$ denote its image in the projective linear group.
We know that $H_1 = G_1$ and $H_2 = G_2$, and we get an isomorphism $ \varphi : H_1/H^1 \to H_2/ H^2$,
which induces a surjective morphism $\tilde{\varphi} : H_1 \to H_2/H^2$.

Assume that $H_1/H^1 \simeq H_2/H^2$ is not abelian. Then $H_2/H^2$ has for quotient 
$PG(\mu)$
and we get a surjective morphism
$\hat{\varphi} : H_1 \onto PG(\la_0+1)$. Let now $\mu \leq \la_0$, and consider
the restriction $\hat{\varphi}_{\mu}$ of $\hat{\varphi}$ to $G(\mu)$. Assume it
is non-trivial. Since the
image of the center is mapped to $1$, it factorizes through an isomorphism
$\check{\varphi}_{\mu} : PG(\mu)\to PG(\la_0+1)$. But this implies that the image
of $\mathcal{B}_n$ inside $G(\mu) \times G(\la_0+1)$ is included
inside $H = \{ (x,y) \ | \ \bar{y} = \check{\varphi}_r(\bar{x}) \}$, where $\bar{x}, \bar{y}$
denote the canonical images of $x,y$.

Let then $\overline{R_{\la}} : B_n \to \PGL(\la)$ denote
the projective representation deduced from $R_{\la}$.
By the
very description of $H$ we have $\overline{R_{\la_0+1}}(b) = \hat{\varphi}( R_{\mu}(b))$ for
all $b \in \mathcal{B}_{n }$, where $\hat{\varphi}$ is the composite
$H_1 \onto H_1/H^1 \stackrel{\simeq}{\to} H_2/H^2 \to \PGL(\la_0+1)$.
Since $\varphi : H_1/H^1 \stackrel{\simeq}{\to} H_2/H^2$ is an isomorphism
and $Z(H_i/H^i)$ is the image of $\F_q^{\times}
\cap H_i$ inside $H_i/H^i$, we have $\hat{\varphi}(\F_q^{\times} \cap H_1) = 1$,
hence $\overline{R_{\la_0+1}}(b) = \check{\varphi} (\overline{R_{\mu}}(b))$ for all
$b \in \mathcal{B}_{n }$, where $\check{\varphi} : H_1 / (\F_q^{\times} \cap H_1) \to \PGL(\la_0+1)$
is the induced morphism.

Note that $H_1 / (\F_q^{\times} \cap H_1) \subset \PGL(\mu)$,
and clearly $\mathrm{Im} \check{\varphi} \supset PG(\la_0+1)$. From this
one deduces that the restriction of $\check{\varphi}$
to $PG(\mu)$ is non-trivial, hence induces an isomorphism $\psi$ between the simple
groups $\psi : PG(\mu) \to PG(\la_0+1)$. Since $\dim \mu \geq 16$ no triality phenomenon can be involved
and thus, up to a possible
linear conjugation of the representations $R_{\mu},R_{\la_0+1}$, we get  (see \cite{WILSON} \S 3.7.5 and \S 3.8) 
that
$\psi$ is either induced by a field automorphism $\Phi \in \Aut(\F_q)$,
or, in case $\la \neq \la'$, by the composition of such an automorphism with $X \mapsto \ ^t X^{-1}$.
In the first case we let $S = R_{\mu}$, in the second case we let $S : g \mapsto ^t R_{\mu}(g^{-1})$.

In both cases, we have $\overline{R_{\la_0+1}}(b) = \Phi(\overline{S}(b)) = \overline{ S^{\Phi}}(b)$ for all
$b \in \mathcal{B}_{n }$, with $S^{\Phi} : g \mapsto \Phi(S(g))$, meaning that the two representations of $\mathcal{B}_{n }$ afforded
by $R_{\la_0+1}$ and $S^{\Phi}$ are projectively equivalent, that is there is $z : \mathcal{B}_{n } \to \F_q^{\times}$
such that $R_{\la_0+1}(b) = S^{\Phi}(b) z(b)$ for all $b \in \mathcal{B}_{n }$. Since $\mathcal{B}_{n }$ is perfect
for $n  \geq 5$ (see \cite{GORINLIN}) we get $z = 1$ ; this proves that the restrictions of $R_{\la_0+1}$ and $S^{\Phi}$
to $\mathcal{B}_{n }$ are isomorphic. In particular, their restrictions to $\mathcal{B}_3$ are isomorphic.
The restrictions of $R_{\la_0+1}$ and $S$ to $\mathcal{B}_3$ are direct sums of the irreducible
representations of the Hecke algebra for $n=3$, restricted to the derived subgroups. There
are three such irreducible representations, of dimensions 1 and 2,
corresponding to the partitions $[3]$,$[2,1]$, $[1,1,1]$.
Note that these restrictions have to contain a constituent
of dimension $2$, for otherwise the image of $\mathcal{B}_3$
would be trivial, 
hence $s_1$ and $s_2$
would have the same image (as $s_1 s_2^{-1} \in \mathcal{B}_3$), which easily implies that the image of $B_{n }$
is abelian, contradicting the irreducibility.

But this implies that
the representation of $\mathcal{B}_3$ associated to $[2,1]$ has to be isomorphic
to its twisted by $\Phi$.
 By explicit computation we get that
the trace of $s_1 s_2 s_1^{-1} s_2^{-1}$
is $1 - (\alpha + \alpha^{-1})$. Since $\F_q = \F_p(\alpha) = \F_p(\alpha + \alpha^{-1})$
this implies $\Phi = 1$.

We thus have $R_{\mu}(b) = S(b)$ for all $b \in \mathcal{B}_{n-1}$.
Note that $S$, viewed as a representation of $B_n$, is isomorphic to $R_{\la}$ for $\la$ equal to either $ \la_0+1$
or possibly to its transpose.
By Lemma \ref{lem:restrBBn}  we get that the only possibility is $\mu = \la_0+1$, since $\mathcal{E}_n$
contains $\la_0+1$ hence not its transposed if different. But this is a contradition which 
proves that each $\hat{\varphi}_{\mu}$ is trivial, hence so is $\hat{\varphi}$, and this
contradicts its surjectivity. Therefore, $H_1/H^1 \simeq H_2/H^2$ is abelian.
It follows that each $H^i$ contains the commutator subgroup of $H_i$. Since both of the
$H_i$ are perfect we get the conclusion by induction on $\la_0$.

\medskip

We now explain how to adapt the proof to the `unitary' case, that is when $\F_p(\alpha+\alpha^{-1}) = \F_{\sqrt{q}} \subsetneq \F_q = \F_p(\alpha)$. We denote $\SU_m(q) \subset \GL_m(q)$ the unitary group associated to
the involutive automorphism of $\F_q$ and recall that $\SU_2(q) \simeq \SL_2(\sqrt{q})$. We also
recall that $\SU_2(q)$ contains a semisimple element of order $1+\sqrt{q}$, namely
$ \begin{pmatrix} \zeta & 0 \\ 0 & \zeta^{-1} \end{pmatrix}$ with $\zeta$ of order $1+\sqrt{q}$,
so that $\bar{\zeta} = \zeta^{-1}$, and we note that $|\SU_2(q)| = \sqrt{q}(q-1)$. We note that
the assumption $o(\alpha) > n$ remains in force. But here we have $\eps(\alpha) = \alpha^{-1}$.
Since $\eps(\alpha) = \alpha^{\sqrt{q}}$ this implies $\alpha^{1+\sqrt{q}}=1$
and therefore $1+\sqrt{q}> n \geq 6$ hence $\sqrt{q} \geq 6$.

First of all, the preliminary analysis of the partitions imply that we can assume
that the image of $\mathcal{B}_{n-1}$ contains a copy of $\SU_2(q)$ acting either on a 2-dimensional
subspace, or on a 4-dimensional subspace via the twisted action $x \oplus x^{-1}$. Therefore,
there is an $x \in G$ originating either from a toric element or from a unitary transvection
of $\SU_2(q)$ such that $\dim (x-1)V = 2$. Moreover, $G$ is tensor-indecomposable by Lemmas \ref{tensindec1}
and \ref{tensindec2}, provided we know that $\SU_2(q)$ contains a semisimple element
of order $>2$ and this holds because $1+ \sqrt{q} > 2$.

Case (ii) is ruled out in a similar way. If $G$ contains a natural $\SU_2(q)$ and therefore
some $g$ with $\dim [g,V] \leq 2$ of order $1+\sqrt{q}$ we conclude as in the non-unitary case.
If $G$ contains instead a twisted-diagonal $\SU_2(q)$, we similarly get an
element $g$ of order $1+\sqrt{q}$ with $\dim [g,V] \leq 4$ providing the same contradiction
as in the non-unitary case, as soon as $1+ \sqrt{q} \geq 7$, which is our assumption here.

For ruling out the monomial case, we assume again $G \subset \F_q^{\times}\rtimes \mathfrak{S}_N$, and
we notice again that $G$ contains some natural or twisted-diagonal $\SU_2(q)$, and one
of its $p$-Sylow subgroups induces as in the classical case an elementary abelian $p$-subgroup $H$ of $\mathfrak{S}_N$
with $\dim [g,V] \leq 2$ for all $g \in H$, but this time of order $\sqrt{q} \geq 6$. This again provides
a contradiction by the same argument.

The argument for the non-monomial imprimitive case applies here verbatim when $p \geq 3$ and,
when $p=2$ we can similarly pick two $\F_2$-linearly independent elements $t_1,t_2 \in G$ originating
from some $p$-Sylow subgroup of $\SU_2(q)$, because we have $\sqrt{q}>2$.

This proves again that $G$ is primitive, tensor-indecomposable,
and we rule out cases (ii) and (iii) of Theorem \ref{theoGURALSAXL}.

Applying theorem \ref{theoGURALSAXL}, we get again that $G$ is a classical group over a subfield $\F_{q'}$
of $\F_q$.

A consequence of lemma \ref{lemmetraces} is that, whenever $\la$ contains a partition $\mu$ of size $n-1$ but not
its transpose $\mu'$, then $G$ contains a natural $\SU_3(q)$ and thus $q' = q$. Otherwise,
we have $\la = \la'$, and therefore $G$ is a subgroup of some $\OSP(\sqrt{q})$.
Moreover, it contains a twisted-diagonal $\SU_3(q)$, and therefore $\F_{q'}$ has
to contain all the $\tr(g) + \overline{\tr(g)}$ for $g \in \SU_3(q)$,
hence all the $\beta + \overline{\beta}$ for $\beta \in\F_q$, that is $\F_{\sqrt{q}}$.

The remaining part of the argument is then completely similar to the first case (and actually easier).

\end{document}